\documentclass[11pt]{amsart}

\usepackage[leqno]{amsmath}
\usepackage{amssymb,amsbsy,amsmath,amsfonts,amssymb,amscd}
\usepackage{latexsym}
\usepackage{graphics}
\usepackage{color}
\input xy
\usepackage{verbatim}
\usepackage{mathrsfs,amsfonts,amsmath,amssymb,epsfig,amscd,xy,amsthm}

\usepackage{amsmath, amssymb}
\usepackage{amsthm} 
\def\Z{{{\mathbb Z}}}
\def\Q{{{\mathbb Q}}}
\def\R{{{\mathbb R}}}

\def\F{{{\mathbb F}}}

\def\GL{\mathrm{GL}}
\def\SL{\mathrm{SL}}
\def\Sp{\mathrm{Sp}}

\theoremstyle{plain} 
\newtheorem{theorem}{\indent\sc Theorem}[section]
\newtheorem{lemma}[theorem]{\indent\sc Lemma}

\newtheorem{proposition}[theorem]{\indent\sc Proposition}

\theoremstyle{definition} 
\newtheorem{definition}[theorem]{\indent\sc Definition}
\newtheorem{remark}[theorem]{\indent\sc Remark}
\newtheorem{example}[theorem]{\indent\sc Example}

\makeatletter
\def\address#1#2{\begingroup
\noindent\parbox[t]{7.8cm}{%
\small{\scshape\ignorespaces#1}\par\vskip1ex
\noindent\small{\itshape E-mail address}%
\/: #2\par\vskip4ex}\hfill%
\endgroup}%
\makeatother

\newcommand\sD{{\mathcal D}}

\newcommand\sF{{\mathcal F}}

\newcommand\sJ{{\mathcal J}}

\newcommand\sK{{\mathcal K}}

\newcommand\sH{{\mathcal H}}

                 \newcommand\sM{{\mathcal M}}

\newcommand\al{\alpha}
\newcommand\be{\beta}
\newcommand\e{\epsilon}

\newcommand\Ga{\Gamma}

\newcommand\ga{\gamma}
\newcommand\de{\delta}

\newcommand{\CC}{\ensuremath{\mathbb{C}}}

\newcommand{\ZZ}{\ensuremath{\mathbb{Z}}}
\newcommand{\QQ}{\ensuremath{\mathbb{Q}}}

\newcommand{\hol}{\ensuremath{\mathcal{O}}}

\newcommand{\PP}{\ensuremath{\mathbb{P}}}

\newcommand{\ra}{\ensuremath{\rightarrow}}

\def\eea{\end{eqnarray*}}
\def\bea{\begin{eqnarray*}}

\newcommand\dual{\mathrel{\raise3pt\hbox{$\underline{\mathrm{\thinspace d
\thinspace}}$}}}
\newcommand\qe{\ifhmode\unskip\nobreak\fi\quad $\Box$}       

\def\BOX{\hfill\lower.5\baselineskip\hbox{$\Box$}}

\newtheorem{theo}[equation]{Theorem}
\newenvironment{rem}{\begin{remark}\rm}{\end{remark}}

\newtheorem{prop}[equation]{Proposition}

\author{Fabrizio  Catanese}

\author{Pietro Corvaja}

\author{Umberto Zannier}

\thanks{AMS Classification:  14D05, 14J29, 14J80, 32S50, 32S20, 20H99, 53D99.\\ 
Key words: Fibrations of algebraic surfaces, number of singular fibres, commutators, mapping class group, symplectic group, symplectic fibrations, stable fibrations. \\
The present work took place in the framework  of the 
 ERC Advanced grant n. 340258, `TADMICAMT' }

\begin{document}

\title[Fibrations and commutators in $\Sp(2g, \ZZ)$]{Fibred algebraic surfaces and commutators in the Symplectic group}

\maketitle

\begin{abstract}
We describe the minimal number of critical points and the minimal number $s$ of singular fibres
for a non isotrivial fibration of a surface $S$ over a curve $B$ of genus $1$,  exhibiting several examples 
and in particular
constructing a fibration with $s=1$ and irreducible singular fibre with $4$ nodes.

Then we consider the associated factorizations in the mapping class group and in the symplectic group.
We describe explicitly which products of transvections on homologically independent and disjoint circles
are a commutator in the Symplectic group $Sp (2g, \ZZ)$.

\end{abstract}

\tableofcontents

\section*{Introduction}

Our present work consists of two tightly related but different parts: the first is geometrical,
and concerns fibrations $f : S \ra B$ of a smooth  complex algebraic surface over a smooth complex curve $B$, with special attention to the case where the base curve $B$ has genus at most $1$.

The second part is of algebraic nature, and determines which powers of products of certain standard 
transvections are a commutator in the symplectic group $Sp (2g, \ZZ)$.

In the first section we begin describing the algebraic version of the so called Zeuthen-Segre formula, relating the
topological Euler-Poincar\'e characteristic $e(S)$ with the number $\mu$ of singularities of the fibres of $f$
(counted with multiplicity).

Then in proposition \ref{extremal} we consider the case where the genus $b$ of the base curve $B$  is $1$,
and describe the cases where $\mu $ is minimal ( $\mu = 3$ or $= 4$). Both cases occur, 
the first case due to the existence of the Cartwright-Steger surface \cite{cs}, the second due to   theorem \ref{product}
of section 3.

We proceed observing that if the base curve has genus $b \geq 2$, then there are non isotrivial fibrations 
without singular fibres;  we also  recall some basic lower bound for the number $s$ of singular fibres of
a moduli stable fibration when the base curve $B$ has genus $ b=0,1$. That $s=1$ occurs for $b=1$ 
was shown by Castorena Mendes-Lopes and Pirola in  \cite{cmlp}  (in their examples the singular fibre is reducible, and either  with fibre genus $g= 9$, or with very high genus,  here we show an example with $g=10$), 
and we use a variant of their method in theorem \ref{product}  to construct an example with $g=9$ and irreducible and nodal singular fibre.

In section four we recall how to such a fibration corresponds a factorization in the Mapping class group $\sM ap_g$,
hence also in the symplectic group $ Sp(2g, \ZZ)$; and we recall several results,  referring to \cite{sy}
for results concerning symplectic fibrations.

The main point is that, if $f : S \ra B$ is such that $b=1$, and the fibre singularities are just nodes,
we get that a product of Dehn twists is a commutator in $\sM ap_g$, respectively
a product of transvections is a commutator in  $ Sp(2g, \ZZ)$.

In the next sections we treat rather exhaustively the purely algebraic question to determine which powers of the standard transvection,
and of the product  of transvections on homologically independent and disjoint circles,
are a commutator in the Symplectic group $Sp (2g, \ZZ)$.

While \cite{sy} state that the product of two Dehn twists cannot be a commutator in  $\sM ap_g$, we show that the
corresponding product of transvections is a commutator in $Sp (2g, \ZZ)$, for all $g \geq 2$.

\section{Fibrations of compact complex surfaces over curves}

\begin{definition}
Let $ f : S \ra B$ be a holomorphic map  of a compact smooth (connected) complex  surface $S$ onto a 
smooth complex curve $B$ of genus $b$.

By Sard's lemma, the fibre $F_y : = f^{-1} (y)$ is smooth, except for a finite number of points $p_1, \dots p_s \in B$
(and then the fibres $F_{p_j}$ are called the singular fibres). 

(1) 
 $f$ is said to be a {\bf fibration} if all smooth fibres are connected (equivalently, all fibres are connected).
 {{} In this case we shall denote by $g$ the genus of the fibres.}

Consider a singular  fibre 
$F_t  = \sum n_i C_i$, where the $C_i$ are distinct  irreducible curves.

(2) Then the {\bf divisorial singular locus of the fibre}  is defined as the divisorial part of the critical scheme,
$ D_t : =  \sum (n_i -1) C_i$, and the {\bf Segre number of the fibre} is
defined as 
$$\mu_t : = deg \sF + D_t K_S  - D_t^2,$$
where the sheaf  $\sF$  is concentrated in the singular points of the reduction $(F_t)_{red}$ of the
fibre $F_t$, and is defined as  the quotient of $\hol_S$ by  the ideal sheaf  generated by the components of the vector 
$d \tau / s$, where $ s = 0 $ is the equation of $D_t$, and where $\tau$ is the pull-back of a 
local parameter
at the point $t \in B$.

More concretely, $$ \tau  = \Pi_j   f_j^{n_j} , s = \tau / ( \Pi_j   f_j),$$ and the logarithmic derivative yields

$$ d \tau = s [  \sum_j n_j (df_j \Pi_{h \neq j} f_h)].$$

\end{definition} 

The following is the algebraic version of the  Zeuthen-Segre formula, expressing how the topological Euler Poincar\'e
characteristic $e(S) $ of $S$, equal to the second Chern class $c_2(S)$ of $S$, differs from the one of a fibre bundle
(for a topological  fibre bundle $ e(S) = 4 (g-1)(b-1)$) {{} (see \cite{cb}, \cite{takagi})}.

\begin{theo}{\bf ( Modern Zeuthen-Segre Formula)}
Let $ f : S \ra B$ be a fibration of a smooth complex surface $S$ onto a curve of genus $b$, and with fibres of genus $g$.

Then $$e (S) = c_2(S) = 4 (g-1)(b-1) + \mu,$$ where $\mu = \sum_{t\in B} \mu_t$,
and $\mu_t $ is the Segre number defined above.

Moreover, $\mu_t \geq 0$, and indeed the Segre number  $\mu_t$ is strictly positive, except if the fibre
$F_t$  is smooth or is a multiple of a smooth
curve of genus $g=1$. 

\end{theo}

 The importance of the formula lies in the fact  that 
 the difference $\mu : = e(S) - 4 (b-1) (g-1)$ is always non negative.
 
 It leads to an interpretation  of $\mu$ as the  total number of singular points of the  fibres, counted with multiplicity,
 in the case where $ g \neq 1$ (observe that if $g=0$, then $S$ is an iterated blow up of a $\PP^1$-bundle over $B$,
 hence $\mu$ is equal to the number of blow ups, and also to the number of singular points on the fibres taken with their reduced structure).
 
 Indeed, if the singularities of the fibres are isolated, then $\mu_t$ is the sum of the Milnor numbers of the singularities; in particular, it equals the number of singular points of the fibre if and only if all the singularities are 
 {\bf nodes}, i.e., critical points where there are local coordinates $(x,y)$ such that locally $f = x^2 - y^2$
 (equivalently, $f = x y$).
 
Most of the times, the formula is used in its non refined form: if $g >1$, then either $\mu>0$, or $\mu=0$ and we have
a differentiable fibre bundle. 

 The formula is well known using topology (see \cite{bpv}), but the algebraic formula is very convenient for explicit calculations. 

Let us look at the particular case where the base curve $B$ has genus $b=1$, hence $e(S) = \mu \geq 0$.
If moreover $g \geq 2$, then we have the following proposition (using also some arguments from \cite{catkeum}):

\begin{theo}\label{extremal}

Let $ f : S \ra B$ be a fibration of a smooth complex surface $S$ onto a curve of genus $b=1$, and with fibres of genus $g \geq 2$.

Then either $e(S) = \mu = 0$, or $e(S) = \mu \geq 3$, equality holding if and only if $S$ is a minimal surface $S$ with $p_g(S) = q(S) = 1$ or $p_g(S) = q(S) = 2$, and with $K^2_S = 9$.
In particular $S$  is then a ball quotient. Moreover, either 

(I) all fibres are reduced, and the singular points of the fibres are either

(I 1) $3$ nodes, or

(I 2) one tacnode ( $f = y^2 - x^4 $ in local coordinates), or

(I 3) one  node and one ordinary cusp ( $f = y^2 - x^3 $ in local coordinates), or

(II) we have one double fibre, twice a smooth curve  of  genus $2$ (hence $g=3$),  plus one node.

If instead $e(S) = \mu = 4 $ and $S$ is minimal, then necessarily either 
\begin{enumerate}
\item
$p_g(S) = q(S) = 1$ or 
\item
$p_g(S) = q(S) = 2$,
or 
\item
{{} $p_g(S) = q(S) = 3$,  and then $g=3$, the fibration has constant moduli and just two singular  fibres, each twice a smooth curve of genus $2$; or  } 
\item
$S$ is a product of two genus $2$ curves {{} in this case $p_g(S) = q(S) = 4$}.

\end{enumerate}
\end{theo}

\begin{proof}
If $S$ is not minimal, every $(-1)$-curve maps to a point, hence $f$ factors as $ p : S \ra S'$, where $S'$ is the minimal model,
and $ f' : S' \ra B$.  Since $e(S) $ equals $e(S')$ plus the number of blow ups,  it suffices to prove the inequality in the  
case where $S$ is minimal.

Since $S$ is non ruled, we have (recall that $\chi(S) = 1 - q(S) + p_g(S)$)  $K^2_S \geq 0, \chi (S) \geq 0$.

By Noether's formula $ 12 \chi(S) = K^2_S + e(S)$, while the Bogomolov-Miyaoka-Yau inequality yields $K^2_S \leq 9 \chi(S)$,
equality holding if and only if $S$ is a ball quotient. Hence $e(S) \geq 3 \chi(S)$; and if $\chi(S) =0$, then necessarily also 
$K^2_S = e(S) = 0$. Otherwise, $e(S) = \mu \geq 6$ unless $\chi(S) = 1$, which  is equivalent to saying that $p_g(S) = q(S) $.

If $e(S)=3$, then $\chi(S)=1$, $K^2_S = 9$ and we have a ball quotient. The map to $B$ shows that $q(S) \geq b = 1$.

On the other hand, the classification of surfaces with $p_g = q$ shows that $p_g = q \leq 4$,
equality holding if and only if $S$ is a product $S = C_1 \times C_2$ of two genus $2$ curves, in which case $K^2=8$
\cite{appendix}.

Moreover (\cite{ccm}, \cite{pirola}, {{} see especially} \cite{hp})  if $p_g=q=3$ either $K^2=6$ or $K^2=8$; {{} hence in the first case $e(S)= 6$, and  in the second case
$e(S)= 4$. In the latter case $S$ is a quotient $(C\times D) / (\ZZ/2)$ where $C$, $D$ are smooth curves of 
respective genera  $2,3$, and the group $\ZZ/2$ acts diagonally with  $B : = C / (\ZZ/2)$ of genus $1$,
and $E : =  D /( \ZZ/2)$ of genus $2$. The only fibrations onto curves of strictly positive genus are the maps to $B$,
respectively $E$. For the map $S \ra B$ all the  fibres are isomorphic to $D$, except two fibres which are the
curve $E$ counted with multiplicity $2$.}

Hence, if $e(S)=3$, then necessarily $p_g=q=1 \ {\rm or } \ p_g=q= 2$. Moreover, since we have a ball quotient, $K_S$ is ample and we claim that $D_t = 0$ for each non multiple fibre (this means that all $n_i$ are equal to $n \geq 2$).

In fact, if $F_t$ is not multiple, $D_t \cdot K_S  =  \sum_i (n_i -1) K_S C_i $, while by Zariski's lemma $D_t^2 < 0$ if we 
do not have a multiple fibre. Since $S$ is a ball quotient, it contains only curves of geometric genus $\geq 2$,
in particular of arithmetic genus $\geq 2$:
if $C_i$ is not a submultiple of $F_t$, then $C_i^2 <0$, hence $K_S \cdot C_i \geq 3$. This obviously contradicts
$\mu \leq 3$.

If a fibre is multiple $F_t = n F'$, then $$D_t K_S = (n-1) K_S F' = (n-1) (2 g' -2),$$
and this contribution is even, and $\leq 2$ if and only if $n=2$ and $g'=2$. $F'$ must be smooth, else its geometric
genus would be $\leq 1$, contradicting that $S$ is a ball quotient. Hence there is only one multiple fibre, and a node on another fibre.

In fact, if a fibre is reduced, then $\mu_t$ is the sum of the Milnor numbers of the fibre singularities.
Each point of multiplicity at least $3$ has  Milnor number at least $4$, so all singularities are 
$A_n$ singularities, i.e., double points, with local analytic equation $y^2 - x^{n+1}$.
Their Milnor number is  equal to $n$.

\end{proof}

\begin{rem}\label{CS}
(a) There exists a ball quotient with $q=p_g=1$: it is the Cartwright-Steger surface \cite{cs}.
For this Rito \cite{rito} asserts that there are exactly $3$ singular fibres, each having a node as singularity.

(b) Ball quotients with $q=p_g=2$ are conjectured not to exist (this would follow from the Cartwright-Steger 
classification if one could prove arithmeticity of their fundamental group).

The claimed proof of this fact in \cite{yeung} is  badly  wrong.  \footnote{ According to two editors of Crelle, the article was indeed withdrawn, but published because of  a technical error of the printer, which has not been publicly acknowledged by the editorial board of Crelle. }

(c) a fibration with $e(S) =4$ and with $S$ a product of two genus $2$ curves is constructed in a forthcoming  section.
{{} For this the  singular points of the fibres are exactly $4$ nodes.} 

  (d) Fibrations with $e(S) =4$ and with $p_g=q=2$ and fibre genus $4,10$  can be obtained using a surface introduced by Polizzi, Rito and Roulleau in 
\cite{prr}, as we now show.

\end{rem}

\begin{theo}\label{prr}
Let $p : S \ra E \times E$ be the degree $4$ map of the Polizzi-Rito-Roulleau surface to the Cartesian square of the Fermat elliptic curve.
Set $\e$ to be  an automorphism of order  three of the Fermat elliptic curve acting on a uniformizing parameter via multiplication with  a primitive third root of $1$, that we also call $\e$.

Then the maps $$\phi_1 , \ \phi_2 : E \times E  \to E, \quad  \phi_1 (x,y) : = x+y, \phi_2 (x,y) : = \e y + x$$
produce fibrations $f_1, f_2 : S \ra E$  with $f_i$ defined as the Stein factorization of $\phi_i \circ p$. They have  the following properties:

\begin{itemize}
\item
$f_2$ has fibre genus $g = 4$, and two singular fibres, each consisting of a smooth genus $2$ curve intersecting
a smooth elliptic curve transversally in two points.
\item 
$f_1$ has fibre genus $g = 10$, and only one singular fibre, having only nodes as singularities, consisting of a smooth genus $6 $ component intersecting two genus one components  transversally each  in respectively two points.
\end{itemize}

\end{theo}

\begin{proof}
Recall that $S$ is constructed in \cite{prr} through a  degree two \'etale map $A \ra E \times E$, corresponding to a character $\chi : \pi_1 (E \times E) \ra \ZZ/2$, and then $S $ is a double cover of the blow up $X$
of $A$ in the two points which are the inverse image of the origin, ramified on the strict transforms of the four elliptic curves
$$ \{ x=0\}, \  \{ y=0\}, \{ y= x \}, \{ y = - \e x\}$$
 which meet pairwise  transversally and  exactly at the origin (observe that the authors use the primitive sixth root of $1$,  $\zeta : = - \e$,  in their notation).

The fibres of $ \phi_1$ intersect these curves in respectively $ 1,1,4,3$ points, those of $ \phi_2$ intersect these curves in respectively $ 1,1,1,3$ points.

Hence the general fibre of $f_1$ has genus $ g$ satisfying $ 2g-2 = 4 \cdot \frac{9}{2} \Rightarrow g = 10$, while 
 the same calculation would seem to show that the general fibre of $f_2$ has genus $ g$ satisfying $ 2g-2 = 4 \cdot \frac{6}{2} \Rightarrow g = 7$: however the fibres of $\phi_2 \circ p$ are, as we shall now show, not connected, and consist of two connected components of genus $g=4$.

The only singular fibre of $f_1$ is the one over $0$, which contains the inverse images of the  two exceptional divisors $D_1$ and $D_2$ and of the
fibre of $\phi_1$ over $0$. 

Blowing up the origin in $E \times E$ we see that the exceptional divisor $D$  intersects each of the four curves in one point,
while the fibre of $\phi_1$ intersects the last two curves in respectively $3,2$ points.

Hence the singular fibre of $f_1$ contains two genus one components which are disjoint and intersect the rest of the fibre
transversally respectively in two points. The rest of the fibre is an irreducible smooth component of genus  equal to $6$, because, as we shall now show, 
 the inverse image in $A$ of the fibre of $\phi_1$ is a connected elliptic curve,
  and then we take a double covering branched in $10$ points.
  
  While the authors take a basis $\zeta e_1, \zeta e_2, e_1, e_2 $ of the lattice $\pi_1 (E \times E)$, we take the more natural basis
  $e_1, \e e_1, e_2, \e e_2$. In the second basis, and using an additive notation, the character $\chi$ takes
  values $(0,1,1,1)$ on the four basis vectors.
  
  It follows then easily that the restriction of the double covering is nontrivial on the five curves
   $$ \{ x=0\}, \  \{ y=0\}, \{ y= x \}, \{ y = - \e x\},  \{ y= - x \},$$
  (and their translates). In particular, the fibres of $\phi_1 \circ p$ are connected.
  
  While, for the fibres of $\phi_2$, like  $\{ \e y +  x = 0 \}$,
  the values of the character $\chi$  on the two periods $ (e_1 - e_2 - \e e_2), (\e e_1 - e_2)$
  are equal to $ 0 + 1 +1 = 0, 1 - 1=0$.
  
  Hence the inverse image of  fibre of $\phi_2$ splits into two components  and we get a fibration  $f_2$ whose fibres are of genus
  $g=4$, $f_2 : S \ra E'$,  where $E'$ is an \'etale cover of degree $2$ of $E$. 
  
  The singular fibres lie over the origin in $E$, hence we get two singular fibres for the two points of $E'$ lying over the origin in $E$. Thus the number of singular fibres $s$ of $f_2$ equals $2$,
  and each singular fibre is the inverse image of the union of an elliptic curve and an exceptional curve ($\cong \PP^1$). Since the double covering of the elliptic curve is branched on $2$ points, while the double cover
  of the exceptional curve is branched in $4$ points, we obtain the desired assertion.

\end{proof}

\section{Number of singular fibres of a fibration}

Once we fix $b,g$ and we consider fibrations $f : S \ra B$ with fibres $F$ of genus  $g$, and genus
 $b$ of the base curve $B$, the  Zeuthen-Segre formula which we have discussed in the previous section gives a relation between 
the the topological Euler characteristic $e(S)$ and the number of singular fibres of $f$, counted with multiplicity.

In particular, if there is only a finite number $c$ of critical points of $f$, it gives an upper bound for the number $c$.

This upper bound must obviously depend on $e(S)$, as shows the case where $b=1$: in fact, in this case there exist unramified coverings $B' \ra B$ of arbitrary degree $m$, and the fibre product 
$$f' : S' : = S \times_B B' \ra B'$$ 
has both numbers $c' = m \cdot c$ and $e(S') = m \cdot e(S)$.

The Zeuthen-Segre formula says also that if there are no singular fibres, for $g \geq 2$, then necessarily 
$e(S) = 4 (b-1)(g-1)$. There are two ways in which this situation can occur (see \cite{takagi} for more details),
since then $S$ is relatively minimal and one can apply  Arakelov's theorem asserting that 
$$K_S^2 \geq 8(b-1)(g-1),$$
equality holding iff all the smooth fibres are isomorphic.

As a consequence, there are two cases when $e(S) = 4 (b-1)(g-1)$: 

{\bf \'Etale bundles:} $K_S^2 = 8(b-1)(g-1),$ and there is a Galois unramified covering $B' \ra B$ such that $$S' : = S \times_B B'  \cong B' \times F,$$

{\bf Kodaira fibrations: } $K_S^2 > 8(b-1)(g-1),$ and not all fibres are biholomorphic.

The only restrictions for Kodaira fibrations are that $b \geq 2, g \geq 3$, and for all such values
of $b,g$  we have Kodaira fibrations.

Assume now that $b \leq 1$, and assume that not all smooth fibres are biholomorphic.
 Let then $$B^* : = \{ t \in B | F_t \ {\rm is \ smooth} \ \}.$$
 
Then the universal cover of $B^*$ admits a non constant holomorphic map into the Siegel space
$\sH_g$, which is biholomorphic to a bounded domain.

The conclusion is that, for $b=1$, there must be at least one  singular fibre, whereas for $b=0$ the number of singular fibres
must be at least $3$.

With the stronger hypothesis that the fibration is moduli stable, i.e., all singular  fibres have only nodes
as singularities and do not possess a smooth rational curve intersecting the other components in two points or less,
one gets a better estimate \cite{bea}, \cite{tan}, \cite{zamora}:

\begin{theo}
Let $f : S \ra B$ be a moduli stable fibration with $g \geq 1$. Then the number $s$ of singular fibres is at least:
\begin{enumerate}
\item
$ s \geq 4$ for $b=0, g\geq 1$,
\item
$ s \geq 5$ for $b=0, g\geq 2$,
\item
$ s \geq 6$ for $b=0, g\geq 3$,
\item
$ s \geq 2$ for $b=1, g = 2$.
\end{enumerate}
\end{theo}

For $b=1$ Ishida \cite{ishida} constructed a Catanese-Ciliberto surface with $g=3$, $K_S^2 = 3$, $p_g=q=1$
having only one singular fibre: but in this case the singular fibre is not a stable curve, it is isomorphic to
the union of $4$ lines in the plane passing through the same point (note that here the Milnor number is $9$, and that for 
a plane quartic curve the number of singular points is at most $6$, so that there is no stable curve with $g=3$
and with $9$ nodes).

Parshin \cite{parshin} claimed that for a moduli stable fibration with $b=1$ one should have $ s\geq 2$,
but the claim was contradicted by \cite{cmlp} who constructed an example with $s=1$ and
with reducible singular fibre.

In the next section, {{} using a variation of the method of \cite{cmlp},} we construct an example where there is only one singular fibre, irreducible and with $4$ nodes
(the number of nodes should be the smallest one, see remark \ref{CS}).

This example will play a role also in the later sections.

\section{A fibration over an elliptic curve with only one singular fibre,  irreducible and nodal}

\begin{theo}\label{product}
There exists  fibrations $f : S \ra B$, where $B$ is a smooth curve of genus $b=1$,
and  the fibres of $f$ are smooth curves of genus $g=9$, with the exception of  a unique singular fibre, which is
an irreducible nodal curve with $4$ nodes. Moreover, $S$ is the product $C_1 \times C_2$
of two smooth genus $2$ curves.

\end{theo} 

\begin{proof}
We achieve the result in three steps.

{\bf Step 1:} we construct, for $i=1,2$, a degree $4$ covering $$f_i : C_i \ra B,$$ branched  only over $O \in B$,
and with $f_i^{-1} (O)$ consisting of two (necessarily simple) ramification points.

{\bf Step 2:} taking $O$ to be the neutral element of the group law
on the genus $1$ curve $B$, we set, as in \cite{cmlp},  
$$S: = C_1 \times C_2, \ {\rm  and  } \ f (x_1, x_2) : = f_1(x_1) -  f_2 (x_2).$$

Hence $x : = (x_1, x_2)$ is a critical point for $f$ if and only if both $x_1$ is a critical point for $f_1$ and
$x_2$ is a critical point for $f_2$.

By the choice made in step 1, we have $4$ such critical points, and for each  of them $f (x)=O$.
Hence $f$ has only one singular fibre $F_O : = f^{-1} (O)$, which possesses exactly  $4$ singular points.

By simple ramification, there is a local coordinate $t$ around $O$, and there are  
local coordinates $z_i$ around $x_i$ 
such that in these coordinates $ f_i(z_i) = z_i^2$. Therefore, at a critical point $x$, there are local coordinates $(z_1, z_2)$ such that
$$ f (z_1, z_2) = z_1^2 - z_2^2,$$ 
and we have a nodal singularity of the fibre $F_O = \{  f (z_1, z_2) = 0\}$.

{\bf Step 3:} we shall show, based on the explicit construction in step 1, that the singular fibre $F_O = f^{-1} (O)$,
which is the fibre product $C_1 \times_B C_2$,  is irreducible.

We observe moreover that the fibre $F_y$ over a point $ y \in B$ is the fibre product of $C_1$ and $C_2$ via the 
respective maps
$f_1 - y$ and $f_2$: hence there are exactly $ 2 \cdot  2 \cdot 4 = 16$ simple ramification points for
the map $F_y \ra B$, hence $ g (F_y) = 9$.

{\bf Construction of step 1:} We construct the two respective coverings $f_i : C_i \ra B$ using Riemann existence's theorem,
and we let $B$ be any elliptic curve, with a fixed point $O$ which we take as neutral element for the group law.

Since the local monodromy at the point $O$ is a double transposition, the monodromy $\mu_i : \pi_1(B \setminus\{O\})$
factors through the orbifold fundamental group
$$ \Ga : =  \pi_1^{orb}(B; 2O) : = \langle \al , \be, \ga |  [\al, \be] = \ga, \ga^2 = 1 \rangle,$$
where $\ga$ represents a simple loop around the point $O$.

We define then two homomorphisms,
$$\mu_1 : \Ga \ra \mathfrak S_4, \mu_1 (\al) : = (1,2,3,4),  \ \mu_1 (\be) : = (1,2) (3,4),$$
$$\mu_2 : \Ga \ra \mathfrak S_4, \mu_2 (\al) : = (1,2,3),  \ \mu_2 (\be) : = (1,4) (2, 3).$$

In both constructions $\mu_i (\be)$ is an element of the Klein group $\sK \cong (\ZZ/2)^2$,
consisting of the three double transpositions and of the identity, as well as $\mu_i (\al \be \al^{-1})= \mu_i (\ga \be).$
We have respectively 
$$\mu_1 (\al \be \al^{-1})= (2,3)(4,1), \ \   \mu_2 (\al \be \al^{-1})= (2,4)(1,3),$$
so that $\mu_1 (\al \be \al^{-1}) \neq \mu_1 ( \be )$ and $\mu_i (\ga)$ is the third nontrivial element in $\sK$, a double transposition as desired.

The conclusion is that $\mu_i(\Ga)$ contains the normal subgroup $\sK$ and is generated by $\sK$ and $\mu_i(\al)$.

Hence $\mu_1 (\Ga) $ is the dihedral group $D_4$, while $\mu_2 (\Ga) $ is the alternating group $\mathfrak A_4$.

We let then $f_i$ to be the degree $4$ covering associated to the monodromy $\mu_i : \Ga \ra \mathfrak S_4$.

{\bf Proof of the assertion of Step 3:} The normalization of the singular fibre of $f$, whch is the fibre product of $C_1$ and $C_2$
over $B$, is a degree $16$ covering of $B$ associated to the product monodromy 
$$\mu_1 \times \mu_2 : \Ga \ra  \mathfrak S_4 \times  \mathfrak S_4 \subset  \mathfrak S_{16}.$$

The irreducibility of this fibre product amounts therefore to the transitivity of the monodromy $ \mu (\Ga) : = \mu_1 \times \mu_2 (\Ga)$
on the product set $\{1,2,3,4\} \times \{1,2,3,4\}.$

Indeed, the $\mu (\al)$-orbit has cardinality $12$ and contains  $ \Sigma : = \{1,2,3,4\} \times \{1,2,3\}.$
Moreover, each element not in $\Sigma$ is of the form $(a,4)$ and $\mu (\be)$ sends $(a,4)$ to an element $(y, 1)$ which lies in
$\Sigma$, whence there is a unique orbit, the monodromy is transitive, and the unique singular fibre $F_O$ of $f$ is irreducible.

\end{proof}
\begin{remark}
Since we  took $B$ to be any elliptic curve, we see that our construction leads to a one-parameter family of
such fibrations. And since a deformation of a product of curves is again a product, we see that any deformation of $f$ which 
has exactly one singular fibre must be as in our construction.
\end{remark}

\section{Fibrations and factorizations in the mapping class group}

Let as usual now $f : S \ra B$ be a fibration of an algebraic surface onto a curve $B$ of genus $b$, such that the fibres 
$F_t$ of $f$ have genus $g$.

As before, we let $B^* $ be the complement of the $s$ critical values $p_1, \dots , p_s$ of $f$. We denote the $s$ singular fibres
by $f^{-1}(p_i) = : F_i$, and set $S^* : = f^{-1}(B^*) = S \setminus (F_1 \cup \dots F_s)$.

Then $f^* : S^* \ra B^*$ is a differentiable fibre bundle, and its monodromy defines homomorphisms
$$  \pi_1 (B^*, t_0) \ra \sM ap_g \ra Sp (2g, \ZZ),$$
where the second homomorphism corresponds geometrically to the bundle $\sJ^*$ of Jacobian varieties 
with fibres $J_t : = Jac(F_t) = Pic^0(F_t)$.

Here $\sM ap_g$ is the Mapping class group $\sD iff^+ (F_0) / \sD iff^0 (F_0)$, introduced by Dehn in \cite{dehn},
and we let
{{} $ \nu :  \pi_1 (B^*, t_0) \ra \sM ap_g$}  be the geometric monodromy.

Fixing a geometric basis, the fundamental group $ \pi_1 (B^*, t_0)$ is isomorphic to the group
$$ \pi_b(s) : = \langle  \al_1, \be_1, \dots \al_b, \be_b , \ga_1 ,\dots, \ga_s | \Pi_1^s \ga_i  \Pi_1^b  [\al_j, \be_j] = 1 \rangle  ,$$
and the image $\de_i : = \nu (\ga_i)$  is a conjugate of the local monodromy around $p_i$.

In the case where the only fibre singularity is a node, then $\de_i$ is a Dehn twist around the vanishing cycle,
a circle $c_i $, whose image in the Symplectic group is the Picard-Lefschetz transvection {{} associated
to the homology class $c$ of $c_i$:}
$$ T_c , \ T_c (v) : = v + (c,v) c, $$
(here $(c,v)$ denotes the intersection pairing on the base fibre $F_0$, a smooth curve of genus $g$).

It is customary to view the monodromy as a factorization
$$ \Pi_1^s \de_i  \Pi_1^b  [\al'_j, \be'_j] = 1    $$
in the Mapping class group (just let $\al_j' : = \nu(\al_j),  \be'_j : = \nu (\be_j)$).

Completing work of Moishezon \cite{moishezon} and  Kas \cite{kas}, Matsumoto  showed the following result
(theorems 2.6 and 2.4 of \cite{matsumoto}):

\begin{theo}\label{fact}
Given a factorization $ \Pi_1^s \de_i  \Pi_1^b  [\al'_j, \be'_j] = 1    $ in the mapping class group $\sM ap_g$, for $g\geq 1$,
there is a differentiable Lefschetz fibration $f : M \ra B$, whose monodromy corresponds to such a factorization,
if and only if the $\de_i$ are negative Dehn twists about an essential simple closed curve.

Moreover, two such fibrations are equivalent, for $g \geq 2$, if and only if the corresponding factorizations are equivalent, via change of a geometric basis in 
$\pi_1(B^*, b_0)$ and via simultaneous conjugation of all  the factors $\de_i, \al'_j, \be'_j$  by a fixed element $a \in \sM ap_g$. 
\end{theo}

In the above theorem a simple closed curve $c$  is said to be essential if it is not the boundary of a disk.
 There are two cases: if its  homology class {{} in $H_1(F_0, \ZZ)$} is non trivial (hence the complementary set is connected)  then  $c$  is said to be 
 nonseparating, or of type I. Else, the complementary set is disconnected, 
 the curve is said to be separating, or of type  II, and pinching the curve to a point one gets the union of two curves of 
 respective genera $h \geq 1$, $(g-h) \geq 1$,
meeting in a point.

\begin{rem}
Matsumoto takes the more restrictive definition in which $M$ is oriented, and that at the  critical points there
are  complex coordinates $z_1, z_2$ such that not only $F$ is locally given by $z_1 z_2$,
but also the complex orientation coincides with the global orientation.
One says then that the Lefschetz fibration is orientable.

Kas does not make this requirement, so there is no requirement imposed on the Dehn twists $\de_i$ 
occurring in the factorization.

An important question is whether a factorization comes from a holomorphic fibration: the case of fibre genus $g=2$
was treated by Siebert and Tian \cite{st}.

A similar question can be posed, requiring $M$ to be a symplectic 4-manifold, and that there is a local symplectomorphism
yielding the local complex coordinates $(z_1, z_2)$ (we take here  the standard symplectic structure on the target $\CC^2$). 
This question was however answered by Gompf \cite{gompf}, see also \cite{abkp}, who showed that any orientable
Lefschetz fibration comes from a symplectic Lefschetz fibration.

\end{rem}

Matsumoto showed, for $g=2$ orientable Lefschetz fibrations, that the number $m$ of singular fibres of type I, and the number 
$n$ of singular fibres of type $II$ satisfy the congruence

$$ m + 2n \equiv 0 \in \ZZ/10.$$

Indeed, the Abelianization of $\sM ap_2$ is isomorphic to $ \ZZ/10.$

We refer to \cite{sy} for more information about the minimal number of singular fibres for an orientable Lefschetz fibration
over a curve of genus $b$, the cases $b=0, 1$ being the open cases.   Stipsicz and Yun state 
that for $b=1$ the number $s$ of singular fibres is at least $3$. The bound would be sharp in genus $g=19$ because of the
Cartwright -Steger surface. Our example in theorem \ref{prr} shows that already in genus $g \geq 4$ we have a product of $4$
Dehn twists which is a commutator.

 In the case $b=1$, the existence of such a  factorization is equivalent to the assertion that a product of 
 $s$ Dehn twists is a commutator in the Mapping class group. 

In view of this, in the next section we focus on a related question, when is the product of certain transvections a commutator
in the $Sp (2g, \ZZ)$.

\section{Commutators in the Symplectic group $Sp (2g, \ZZ),\  g \leq 2$}

\subsection{The case $g=1$}

\medskip
As a warm up, let us begin with the case $g=1$, where $Sp (2, \ZZ) = SL(2, \ZZ)$.

In this case the group surjects to the group $\PP SL(2, \ZZ) $ of integral M\"obius transformation.
It is known, see for instance \cite{serre},  that $\PP SL(2, \ZZ) $ is the free product $(\ZZ/2)* (\ZZ/3)$, where  the first generator comes from the matrix $A $,
the second generator comes from the matrix $B$,
\begin{equation*}
A:=\left(\begin{matrix}
0&-1\cr 1&0
\end{matrix}\right), \ 
B:=\left(\begin{matrix}
0&1\cr -1&1
\end{matrix}\right).
\end{equation*}

We consider the standard transvection $T= T_{e_1}$, with matrix
 \begin{equation*}
T:=\left(\begin{matrix}
1&1\cr 0&1
\end{matrix}\right),
\end{equation*}
giving rise to the projectivity $ z \mapsto z+1$.

\begin{prop}

(i) One has $T^{-1}= AB$, hence the image of $T^{-1}$ in the Abelianization  $(\ZZ/2)\times  (\ZZ/3) \cong  (\ZZ/6)$
of $\PP SL(2, \ZZ) $ is equal to $(1,1) \equiv 1\in \ZZ/6$, and no power  $T^m$, for $ m $ not divisible by $6$, is a product of commutators.

(ii) $T^{2m}$ is a commutator in $GL(2, \ZZ)$ $\forall m$.

(iii) $T^{m}$ is  a product of commutators in $GL(2, \ZZ)$ only if $ m$ is even.

(iv) No power $T^m$, $m \neq 0$, is a commutator in $SL(2, \ZZ)$.

\end{prop}

\begin{proof}
An immediate calculation shows that  $T^{-1}= AB$, hence assertion (i) follows.

(ii) follows by taking
 \begin{equation*}
C:=\left(\begin{matrix}
-1&0\cr 0&1
\end{matrix}\right),
\end{equation*}
hence $$C T^m C^{-1} (T^m)^{-1} = T^{-2m}.$$

(iv) assuming that $$T^m = X^{-1} Y X Y^{-1}, \ m \in \ZZ, m \neq 0,$$
equivalently $$ Z : =  X^{-1} Y X = T^m Y.$$

Setting 
 \begin{equation*}
Y:=\left(\begin{matrix}
r&s\cr t&u
\end{matrix}\right),
\end{equation*}

we get 
 \begin{equation*}
Z:=\left(\begin{matrix}
r + m t&s + mu\cr t&u
\end{matrix}\right).
\end{equation*}
Since $Y$ and $Z$ are conjugate, they have the same trace, hence:
$$ r + m t + u = r + u \Rightarrow t=0. $$ 
Since $Y$ has determinant $=1$, we obtain $$ ru =1 \Rightarrow r= u= \pm 1,$$
and possibly replacing $Y$ with $-Y$ we obtain $$r=u=1\Rightarrow Y = T^s.$$

Whence, $$T^s X = X T^{m+s},$$
hence $ X (e_1)$ is an eigenvector for $T^s$; since $s\neq 0$, $X$ is also upper triangular,
hence a power of $T$ and we reach a contradiction.

(iii) we observe that  reduction modulo $2$ yields a projection $GL_2(\Z)\to GL_2(\Z/2\Z)\simeq \mathfrak{S}_3$ and that $T$ is sent by this projection to a transposition, hence an  odd permutation.

\end{proof}

The fact that the condition of being a commutator changes drastically, if one allows  orientation reversing  transformations,
occurs also in higher genus. For instance Szepietowski \cite{szep} proved:

\begin{theo}
Let $c$ be an essential closed circle in a compact complex curve $X$ of genus $g \geq 3$: then any power 
of the Dehn twist $\de_c$ is a commutator in the extended mapping class group 
$$\sM ap^e_g = 
\sD iff (X) / \sD iff^0(X).$$

\end{theo}

\subsection{The case $g=2$.}
We consider the lattice $\Z^4$ with its canonical basis $e_1,e_2,e_3,e_4$,
and define the symplectic form $(\cdot|\cdot)$ on $\Z^4$  by setting
\begin{equation*}
(e_1|e_2)=1= (e_3|e_4)
\end{equation*}
and 
\begin{equation*}
(e_i|e_j)=0,\qquad \mathrm{for} \{i,j\}\not\in\{\{1,2\},\{3,4\}\}.
\end{equation*}
The matrix of this symplectic form is then
\begin{equation*}
J_2:=\left(\begin{matrix}
0&1&0&0\cr -1&0&0&0\cr 0&0&0&1\cr 0&0&-1&0
\end{matrix}\right).
\end{equation*}
We denote by $ Sp (4, \ZZ)$ the corresponding symplectic group, i.e. the group of $4\times 4$ matrices $X$ with integral coefficients satisfying
\begin{equation*}
{}^t X\cdot J_2\cdot X= J_2.
\end{equation*}
Let now  $T\in\Sp_4(\Z)$ be the matrix 
\begin{equation*}
T=\left(\begin{matrix}
1&1&0&0\cr 0&1&0&0\cr 0&0&1&0\cr 0&0&0&1
\end{matrix}\right)
\end{equation*}
We prove in this subsection the following result

\begin{theorem}\label{T.uno}
Let $m\in\Z$ be an integer; the power $T^m$ of $T$ is a commutator in the group $Sp_4(4,\Z)$ if and only if $m$ is even. In particular, $T$ itself is not a commutator.
\end{theorem}

 One direction of the equivalence will follow from reduction modulo 2; we shall prove  the following result
 
 \begin{theorem}\label{T.due}
 The matrix $T$ is not a commutator in the group $Sp(4,\F_2)$.
 \end{theorem}

The reverse implication in Theorem \ref{T.uno} shall result from a simple explicit construction given below. Note that as a consequence of this latter implication,  every power of $T$ is a commutator in $Sp_4(4,\F_p)$ for every odd prime $p$.\medskip

\bigskip

{\bf Notation}. For a vector space $V$ and vectors $v_1,\ldots,v_k\in V$, we denote by $<v_1,\ldots,v_k>$ the sub-vector space generated by $v_1,\ldots,v_k$. When $\Lambda$ is a lattice (or a $\Z$-module), and $v_1,\ldots,v_k$ are $k$ elements of $\Lambda$, we denote by the same symbol $<v_1,\ldots,v_k>$ the $\Z$-module generated by $v_1,\ldots,v_k$, when no confusion can arise.

For a vector $v\in \Z^4$ (or more  generally  in a module  with symplectic form $(\cdot|\cdot)$), the symbol $v^\perp$ denotes its orthogonal with respect to the given symplectic form.

\medskip

We begin by proving Theorem \ref{T.due}.  
We start with the following Proposition, which in fact holds in arbitrary characteristic  not dividing the integer $m$ appearing in the statement, and is one of the main tool in the proof of the Theorem.  

\begin{proposition}\label{T.proposition}
Let $m\neq 0$ be an integer and suppose that  
\begin{equation}\label{E.Commutatore}
T^m=XYX^{-1}Y^{-1}
\end{equation}
 for two matrices $X,Y\in Sp(4,\Z)$. Then either $e_1$ is an eigenvector for both $X$ and $Y$, or the orbit of $e_1$ under the subgroup generated by $X$ and $Y$ is contained in  a two dimensional  sub-lattice of $\Z^4$, contained in $e_1^\perp$ and invariant under  $X$ and $Y$.  
\end{proposition}


\noindent{\tt Remarks}: (1) The above proposition could be extended in higher dimensions: for the analogue in dimension $2g$, the result would be that the orbit of $e_1$ under the group generated by $X$ and $Y$ would be contained in an invariant  subgroup $\Lambda\subset\Z^{2g}$,  satisfying $\Lambda\subset\Lambda^{\perp}\subset e_1^\perp$. (2) We have stated the proposition over the integers, but we could have worked over any field (of characteristic not dividing $m$); in that case we would speak of sub-vector spaces instead of sub-lattices. 

\begin{proof}
Let us put $\Delta:T-I$, where $I=I_4$ is the identity matrix. Note that $\Delta^2=0$ and that $\Delta v=0$ for each $v\in e_1^\perp$.
Also, $T^n=I+n\Delta$, for all $n\in\Z$.

We note the useful equality
\begin{equation*}
e_1^\perp=<e_1,e_3,e_4>=\ker \Delta
\end{equation*}
(where we identify the matrix $\Delta$ with the multiplication-by-$\Delta$ endomorphism of $\Z^4$).
We shall also keep in mind that $\Delta\cdot\Z^4=<e_1>$.
\smallskip

We can rewrite equation \eqref{E.Commutatore} in the form
\begin{equation}\label{E.CommDelta1}
XY-YX=m\Delta YX,
\end{equation}
as well as
\begin{equation}\label{E.CommDelta2}
X^{-1}Y^{-1}- Y^{-1}X^{-1}= m Y^{-1}X^{-1}\Delta.
\end{equation}
It immediately follows from the first of the two identities that $\mathrm{Tr}\Delta YX=0$, which means precisely that the coefficient on the first column - second row of $YX$ vanishes. This property  can be stated as 
\begin{equation*}
YXe_1\in e_1^\perp.
\end{equation*}
Also, interchanging $X,Y$ turns their commutator into its inverse $T^{-m}=I-m\Delta$ and repeating the argument we also obtain
\begin{equation*}
YXe_1\in e_1^\perp.
\end{equation*}

Again using equation \eqref{E.Commutatore} one gets
\begin{equation*}
XYX^{-1}-Y=m\Delta Y
\end{equation*}
and, noting that interchanging  $X,Y$ turns their commutator $T^m$ into $T^{-m}$,
\begin{equation*}
YXY^{-1}-X=-m\Delta X.
\end{equation*}
From these relations we obtain as above that $\mathrm{Tr}\Delta X= \mathrm{Tr}\Delta Y=0$, i.e. $Xe_1\in e_1^\perp,\, Ye_1\in e_1^\perp.$
Summarizing we have
\begin{equation*}
Xe_1,Ye_1,XYe_1,YXe_1\in e_1^\perp.
\end{equation*}
We now notice that the commutator $XYX^{-1}Y^{-1}$ does not change if we replace $X$ by $XZ$, where $Z$ commutes with $Y$, or $Y$ by $YZ$, where $Z$ commutes with $X$; hence taking for $Z$ any power of $Y$ in the first case and any power of $X$ in the second case, we also obtain that for each $n\in\Z$,
\begin{equation}\label{E.inclusioni1}
 XY^n e_1, \, YX^n e_1   \in e_1^\perp.
\end{equation}

The  identity \eqref{E.CommDelta2} implies, considering that $\ker\Delta=e_1^\perp$, that 
\begin{equation*}
X^{-1}Y^{-1} v= Y^{-1}X^{-1} v
\end{equation*}
 for each $v\in e_1^\perp$.

\bigskip
Also, observing that all monomials in $X,Y$ are symplectic matrices, and that for every symplectic matrix $F$ the relation $(Fe_1| e_1)=0$ implies $(F^{-1}e_1|e_1)=0$, we also obtain that for all $n\in\Z$:
\begin{equation}\label{E.inclusioni2}
X^nY^{-1} e_1, \,  Y^n X^{-1} e_1   \in e_1^\perp.
\end{equation}

We now pause to prove the following 

{\tt Claim}: {\it The orbit of $e_1$ under $X$ (resp. under $Y$) is contained in a proper sub-vector space of $\Q^4$.}

{\tt Proof of the Claim}. This follows from the relations $YX^n e_1\in e_1^\perp$, included in \eqref{E.inclusioni1} and valid for all $n\in \Z$,  which imply that  the orbit of $e_1$ under $X$ is included in the hyperplane $Y^{-1} (e_1^\perp)=(Y^{-1}e_1)^\perp$. Of course, the relations $XY^n e_1\in e_1^\perp$ imply the same conclusion for the orbit under $Y$. This proves our claim.

\smallskip

We now prove that such sub-spaces must be one or two-dimensional:

{\tt Claim}: {\it The vector space generated by the orbit of $e_1$ under $X$ (resp. under $Y$) cannot be three-dimensional.}

{\tt Proof of the Claim}. Suppose by contradiction that such a vector space has dimension $3$. Then it admits the base  $(X^{-1}e_1, e_1, Xe_1)$. Since  the three vectors $X^{-1}e_1, e_1, Xe_1 $ all belong to $e_1^\perp$, as we have seen in \eqref{E.inclusioni1}, \eqref{E.inclusioni2}, this vector space must coincide with $e_1^\perp$, which then is an invariant subspace for $X$. But if  a symplectic operator leaves invariant a subspace, it also leaves invariant its orthogonal, so in this case the line $<e_1>$  would be $X$ -invariant, contrary to our assumption  that $e_1,Xe_1, X^{-1}e_1$ are linearly independent.
\medskip

Then only three cases must be considered for the proof of the proposition:\smallskip

(1) The two orbits of $e_1$ (under $X$ and $Y$) are contained in a line; this line is then $<e_1>$ and in this case the assertion of the proposition is plainly verified.\smallskip

(2) The vector $e_1$ is an eigenvector for $X$ and its orbit under $Y$ is contained in a plane $W$; in this case we must  show that $W$ is contained in $e_1^\perp$ and that it is $X$-invariant. Of course, the symmetric situation, when $e_1$ is an eigenvector only of $Y$, is treated in exactly the same way. \smallskip

(3) The two orbits generate planes $W_X,W_Y$. In this case we must show that $W_X=W_Y$ and that this common plane is contained in $e_1^\perp$.\smallskip

\smallskip

Let us consider now the second case: $Xe_1=\pm e_1$ and the orbit of $e_1$ under $Y$ generates a plane $W=<e_1,Ye_1>=<e_1,Y^{-1}e_1>$. Since $X^{-1}Y^{-1}$ and $Y^{-1}X^{-1}$ coincide in $e_1$ and $Xe_1=\pm e_1$, we   have $X^{-1}Y^{-1}e_1=\pm Y^{-1}e_1$, so both $e_1, Y^{-1}e_1$ are eigenvectors for $X$, so $W$ is $X$-invariant. Since $Y^{-1}e_1\in e_1^\perp$, the inclusion $W\subset e_1^\perp$ holds, and  the verification of the proposition in this case is complete.
\smallskip

In the last case to examine, let 
\begin{equation*}
W_X=<e_1,Xe_1>=<e_1,X^{-1}e_1>, \qquad  W_Y=<e_1,Y e_1> = <e_1,Y^{-1}e_1>
\end{equation*}
and again note that $W_X\subset e_1^\perp, W_Y\subset e_1^\perp$. If $W_X\neq W_Y$, then the subspace generated by $W_X$ and $W_Y$ would coincide with the hyperplane $e_1^\perp$  and would be generated by $e_1, Xe_1, Ye_1$. Now,  since $XYe_1\in e_1^\perp$, we would obtain that $e_1^\perp$ is $X$-invariant, so again $e_1$ would be an eigenvector for $X$, contrary to our assumptions. So we cannot have $W_X\neq W_Y$ and the proposition is proved in this last case too.

\end{proof}

Thanks to Proposition \ref{T.proposition}, we can divide the proof of Theorem \ref{T.due} into two cases, according to the orbit of $e_1$ under $X,Y$ being a line or a plane.

\medskip

\noindent{\tt First case: $e_1$ is an eigenvector for both $X$ and $Y$}.
Note that in this case the three dimensional vector space $e_1^\perp=<e_1,e_2,e_3>$ is also invariant. 
Also, since the rational eigenvalues of $X,Y$ must be $\pm 1$, we can suppose that $Xe_1=e_1=Ye_1$ (we can change sign to $X$ and to $Y$ without changing their commutator $XYX^{-1}Y^{-1}$). This fact, plus the fact that   $e_1^\perp$ is invariant under $X$ and $Y$ means that the two matrices are of the form
\begin{equation}\label{E.FormaMatrici}
\left(\begin{matrix}1&*&*&*\cr 0&1&0&0\cr 0&*&*&*\cr 0&*&*&*
\end{matrix}\right)
\end{equation}

Put $\Delta=T-I$, where $I=I_4$ is the identity matrix. Recall that $\Delta$ is nilpotent, satisfying $\Delta^2=0$ 
which implies that for every $m\in\Z$,
\begin{equation*}
T^m=I+m\Delta
\end{equation*}

 We take a break  to prove the following

\begin{lemma}\label{T.Delta}
Under the assumption that $Xe_1=Ye_1=e_1$, we have
\begin{equation*}
X\Delta=Y\Delta=\Delta\qquad \mathrm{and} \, \Delta X=\pm \Delta,\, \Delta Y=\pm\Delta.
\end{equation*}
\end{lemma}

\begin{proof}
Since the sub-group $\Delta\cdot \Z^4$ is generated by $e_1$ and $X,Y$ fix $e_1$, we have the first equalities. The second ones follow from the fact that the three rank-one endomorphisms $\R^4\ni v\mapsto \Delta v$,  $\R^4\ni v\mapsto \Delta X  v\in\R^4$ and $\R^4\ni v\mapsto \Delta Y  v\in\R^4$ share the same kernel  (coinciding with the hyperplane $e_1^\perp$) and the same image    (i.e. the line $<e_1>$). Hence there exists a non-zero scalar $\lambda$ such that $\Delta X=\lambda \Delta$ (and analogously for $\Delta Y$); clearly $\lambda$ is an integer. We claim that it has no prime divisors: actually any prime $p$ dividing $\lambda$ would give the congruence $\Delta X\equiv 0 \pmod p$, which cannot hold since $X$ is invertible modulo $p$ and $\Delta$ has rank one modulo $p$. Hence $\lambda=\pm 1$, so $\Delta X=\pm \Delta$ and the same must hold for $\Delta Y$.
\end{proof}

\begin{lemma}
Let $X,Y\in Sp(4,\Z)$ be of the form \eqref{E.FormaMatrici} with $XYX^{-1}Y^{-1}=T^m$. Then
\begin{equation}\label{E.Commutatore2}
XY-YX=\pm m\Delta.
\end{equation}
\end{lemma}
 
 \begin{proof}
 From the relation   \eqref{E.Commutatore} one immediately derives
 \begin{equation*}
 XY-YX= m\Delta YX.
 \end{equation*}
 Applying twice the preceding lemma, we have $\Delta YX= \pm \Delta X =\pm \Delta$. 
 \end{proof}

Our aim now is proving that  the relation \eqref{E.Commutatore2} cannot hold for any odd integer $m$.  This will follow from an argument modulo $2$, leading to the next proposition (where $\Delta$ will denote the reduction of the previous matrix $\Delta$ modulo $2$): 

\begin{proposition}
The equation 
\begin{equation}\label{E.Commutatore3}
XY-YX=\Delta
\end{equation} admits no solution in matrices $X,Y\in Sp(4,\F_2)$ satisfying \eqref{E.FormaMatrici}.
\end{proposition}

\noindent{\tt Remark}. Of course, the above proposition concerns an explicitely given finite group, so it might be proved by exausting all possible cases. However, we want to prove it by developing   some general arguments which could be useful also in higher dimension.

\smallskip

Let us suppose to have a solution $(X,Y)$ to $XY-YX=\Delta$ in $Sp(4,\F_2)$ of the form \eqref{E.FormaMatrici}. We can apply the preceding lemmas, which hold {\it a fortiori} modulo $2$, and deduce in particular 
  from Lemma \ref{T.Delta}   that $\Delta$ commutes with $X$ and $Y$, so also the symplectic matrix $T=I+\Delta$ commutes with $X,Y$. Replacing if necessary $X$ by $TX$, which does not change the commutator, we can suppose that the coefficient on the first line - second column on $X$ vanishes. We can suppose the same for $Y$, so $X,Y$ will both be of the form
\begin{equation}\label{E.FormaMatrici2}
\left(\begin{matrix}1&0&*&*\cr 0&1&0&0\cr 0&*&*&*\cr 0&*&*&*
\end{matrix}\right).
\end{equation}
 The following lemma ensures that we can basically choose the form of the second column too.

\begin{lemma}
Let  $k$ be any field and  $(X,Y)$  be a solution to the equation \eqref{E.Commutatore3} with $X,Y\in 
 Sp (4, k)$ of the above form \eqref{E.FormaMatrici2}. 
Then $Xe_2\neq e_2$ and $Ye_2\neq e_2$. Also, $Xe_2\neq Ye_2$.
\end{lemma}

\begin{proof}
Suppose by contradiction that $Xe_2=e_2$ (the argument is symmetrical if $Ye_2=e_2$). Then, since the plane $<e_1,e_2>$ is invariant by multiplication by $X$, the same must be true of its orthogonal, which is $<e_3,e_4>$. Now, write $Ye_2=e_2+v$, where $v\in<e_3,e_4>$ (this is certainly possible since $Y$ is of the  form \eqref{E.FormaMatrici2}). The relation \eqref{E.Commutatore3} applied to the vector $e_2$ gives
\begin{equation*}
e_1=\Delta e_2=XYe_2-YXe_2=X(e_2+v)-Y e_2 =e_2+Xv-e_2-v=Xv-v,
\end{equation*}
which is impossible since $v$ and $Xv$ belong to the plane $<e_3,e_4>$.
This proves the first two inequalities. Suppose now $Xe_2=Ye_2$. Then from \eqref{E.Commutatore3} applied to the vector $e_2$ we obtain, writing $Xe_2=e_2+v=Ye_2$, with $v\in <e_3,e_4>$,
\begin{equation*}
XYe_2-YXe_2=Xv-Yv=e_1.
\end{equation*}
But from $Xe_2=e_2+v=Ye_2$ and $(e_2|v)=0$ we obtain 
\begin{equation*}
(e_2+v|Xv) = (e_2 +v| Yv) =0,
\end{equation*}
so $(e_2+v|Xv-Yv)=0$ which contradicts $Xv-Yv=e_1$ (since $(e_1|v)=0$ and $(e_1|e_2)=1$).
\end{proof}

Let us now go back to characteristic $2$.
Thanks to the above lemma and the form \eqref{E.FormaMatrici2} for $X$ we can  write $Xe_2=e_2+w$ for some non-zero vector $w\in <e_3,e_4>$.
Also, again by the above lemma, $Ye_2=e_2+w'$ for some vector $w'\neq w$ in the plane $<e_3,e_4>$. Since $w,w'$ are distinct non zero vector in the plane $<e_3,e_4>$, necessarily $(w|w')=1$, 
 so we can suppose without loss of generality that $w=e_3$ and $w'=e_4$. Then, remembering that $X,Y$ are symplectic, we deduce that they   take  the form  
\begin{equation}\label{E.FormaMatriceX}
X=\left(\begin{matrix}1&0&a&c\cr 0&1&0&0\cr 0&1&b&d\cr 0&0&a&c
\end{matrix}\right), \qquad 
Y=\left(\begin{matrix}1&0&e&g\cr 0&1&0&0\cr 0&0&e&g\cr 0&1&f&h
\end{matrix}\right)
\end{equation}
for some scalars $a,b,c\in\F_2$ with $ad\neq bc$ and $eh\neq fg$.
But then, applying once again the relation \eqref{E.Commutatore3} to the vector $e_2$ we obtain $Xe_4-e_4=Ye_3-e_3+e_1$, i.e.
\begin{equation*}
\left(\begin{matrix} c\cr 0\cr d\cr c-1\end{matrix}\right)=
\left(\begin{matrix} e+1\cr 0\cr e-1\cr f\end{matrix}\right).
\end{equation*}
 Since $e+1=e-1$, the above equality implies  $c=d=e+1\neq e$,  and $c=f+1$, so $f=e$, so either $c=d=0$ or $e=f=0$ which contradicts the non-vanishing of the two determinants $ad-bc$ and $eh-fg$.

This achieves the proof in the first case.

\medskip

\noindent{\tt Second case: a plane containing $e_1$ and contained in $e_1^\perp$  is invariant under $X$ and $Y$.}

Without loss of generality, we can suppose that this plane is $<e_1,e_3>$. It is more convenient to write the matrices with respect to the ordered basis $(e_1,e_3,e_2,e_4)$. With respect to this new ordered basis, the symplectic form is expressed by the matrix
\begin{equation*}
\left(\begin{matrix} 0&I\cr -I&0
\end{matrix}\right)
\end{equation*}
where $I=I_2$ denotes the $2\times 2$ identity matrix and $0$ the $2\times 2$ null-matrix.

The conjugate matrices, still denoted by $X,Y$, will take the form
\begin{equation*}
X=\left(\begin{matrix} A&AR\cr 0&{}^tA^{-1}
\end{matrix}\right),\qquad 
Y=\left(\begin{matrix} B&BS\cr 0&{}^tB^{-1}
\end{matrix}\right)
\end{equation*}
for two matrices $A,B\in\GL_2(\Z)$ and symmetric matrices $R,S$ (with integral coefficients).
The matrix corresponding to $T$ in this new basis is
\begin{equation*}
{T'}:=\left(\begin{matrix} I&E\cr 0&I
\end{matrix}\right)
\end{equation*}
where $E=\left(\begin{matrix} 1&0\cr 0&0
\end{matrix}\right)$.
Now the condition $XYX^{-1}Y^{-1}={T'}$ is equivalent to 
$XY-YX=(T'-I)YX$, which amounts to the two conditions
\begin{equation}\label{E.Matrix}
AB=BA,\qquad ABS+AR({}^tB^{-1})-BAR-BS({}^tA^{-1})=E({}^tB^{-1})({}^tA^{-1}).
\end{equation}
Now we prove that:

{\it  The above equation has no solution $(A,B,R,S)$ with $A,B\in\GL_2(\F_2)$ and $S,T$ symmetric (with coefficients in $\F_2$).}

To prove this claim, let us rewrite the second equality, after using the commutativity of $A,B$, as
\begin{equation*}
AB\left(S+B^{-1}R({}^tB^{-1})-R - A^{-1}S({}^tA^{-1})\right)=E({}^tB^{-1})({}^tA^{-1}).
\end{equation*}
 Observe that the right-hand side has rank one.
We  then conclude via the following lemma, which implies that the symmetric matrix inside the parenthesis cannot have rank one:

\begin{lemma}
Let $T=\left(\begin{matrix} a&b\cr b&c\end{matrix}\right)$ be a $2\times 2$ symmetric matrix with coefficients in $\F_2$. Let   $X\in\GL_2(\F_2)$ be an invertible matrix. Write $X\cdot T\cdot {}^tX$ as $ \left(\begin{matrix} a'&b'\cr b'&c'\end{matrix}\right)$. Then
\begin{equation*}
a+b+c=a'+b'+c'.
\end{equation*}
In particular,  every linear combination of symmetric matrices of the form $T-X\cdot T\cdot {}^tX$ has rank zero or two.
\end{lemma}

\begin{proof}
Recall that a two-dimensional vector space over $\F_2$ contains exactly three non-zero vectors $v_1,v_2,v_3$, and that their sum vanishes. To a symmetric matrix $T=\left(\begin{matrix} a&b\cr b&c\end{matrix}\right)$ corresponds a symmetric bilinear form $(\cdot |\cdot)$ on $\F_2^2$. The quantity $a+b+c$ equals   the sum 
\begin{equation*}
(v_1|v_1)+(v_2|v_2)+(v_1|v_2)=(v_1|v_2)+(v_2|v_3)+(v_3|v_1)
\end{equation*}
 which is invariant under permutations of $v_1,v_2,v_3$, i.e. under transformations $T\mapsto XT{}^tX$. 
\end{proof}

\medskip

{\tt Proof of Theorem \ref{T.uno}}. If $m$ is an odd integer, then $T^m\equiv T \pmod 2$ and by Theorem \ref{T.due} $T^m$ cannot be a commutator.

Let now $m$ be even. We look for a solution $X,Y\in Sp(4,\Z)$  of the second type, i.e. with $<e_1,e_3>$ invariant by $X,Y$. We have just seen that we reduce to the matrix equation \eqref{E.Matrix}, where now $E$ is replaced by $mE$. We can find a solution by taking $R=0, B=I$ and reducing to
\begin{equation*}
AS({}^tA)-S=mE,
\end{equation*}
which is  solvable for  every even $m$, e.g. by taking $A=\left(\begin{matrix} 1&1\cr 0&1\end{matrix}\right)$ and
$S=\left(\begin{matrix} 0&m/2\cr m/2&0\end{matrix}\right)$.
%

\section{Alternative proof for $g=2$}

We give an alternative proof of Theorem \ref{T.due} (hence of Theorem \ref{T.uno}). This  proof makes use of the isomorphism between the group $Sp(4,\F_2)$ and the symmetric group $\mathfrak{S}_6$. 

{ This isomorphism has a nice and classical geometric interpretation, which we now briefly describe, in the spirit of the first part of our work.

Recall that every algebraic curve $C$ of genus $g=2$ has a canonical map which is a double covering of the projective line branched in six points, so that
there is an involution on $C$, called  the hyperelliptic involution, whose six fixed points  $P_1, \dots, P_6$ are called the
  Weierstrass points (they  are the critical points for the canonical map). 
 
 Hence every curve $\mathcal{C}$ of genus $2$ admits an affine  model of equation
 \begin{equation*}
y^2=(x-\alpha_1)\cdots (x-\alpha_6)
\end{equation*} 
for pairwise distinct complex numbers $\alpha_1,\ldots,\alpha_6$ (thus $P_i = (\al_i, 0)$).
 
 Given any fibration $S^*\to B^*$  in curves of genus $2$ and a point $b_0 \in B^*$, the action of the fundamental group $\pi_1(B^*,b_0)$ on the fibre of $b_0$ gives, as described above, a morphism   $\pi_1(B^*, b_0) \to \sM  ap_2$; this morphism also induces a permutation of the six Weierstrass points, hence a representation $\pi_1(B^*, b_0)\to \mathfrak{S}_6$. 
 
 On the other hand, the first homology group $H_1(C, \ZZ/2)$ is isomorphic to the subgroup  $Pic^0(C)[2]$ of the 2-torsion points in the Jacobian variety $Jac(C) \cong Pic^0(C)$.
 
 This subgroup is  isomorphic to $(\ZZ /2)^4$.  Since $ 2P_i \equiv 2 P_j \equiv K_C$  (here $K_C$ is the (degree two) canonical divisor of $C$),  and since   $  div (y) \equiv  \sum_i P_i - 3 K_C$, it has a basis given by the differences $P_1- P_2,  P_2 - P_3,  P_3-P_4, P_4 - P_5$ (indeed,  $\sum_i P_i = \
 div (y) \equiv  3 K_C \Rightarrow (P_1- P_2 ) + (P_3-P_4) + (P_5-P_6) \equiv 0$).

 The  morphism $\sM  ap_2\to Sp(4,\Z)$ can be composed with reduction modulo $2$, thus  giving a homomorphism $\sM  ap_2 \to Sp(4,\Z/2\Z)$,
 where the symplectic form modulo 2 is called the Weil pairing on the group  $Pic^0(C)[2]$, and corresponds to cup product in cohomology.

 To see that   the two groups $Sp(4,\Z/2\Z)$ and $\mathfrak{S}_6$ are indeed isomorphic, we observe that the half twist on a simple arc joining 
 $\al_i$ and $\al_j$, which yields a transposition exchanging  the two points $P_i, P_j$, lifts to a Dehn twist $\de_{i,j}$ which maps
  to a transvection $T_{i,j}$ on the class corresponding to
 $P_i - P_j$. 
 
 Hence  we have defined a homomorphism  $\mathfrak{S}_6 \hookrightarrow \GL (H_1(C, \ZZ/2\Z)) $, which is an embedding
 because  a permutation fixes all the  basis vectors if and only if it is the identity. Moreover, the two groups have the same cardinality, hence we have an isomorphism.

The alternative approach  carried out in this paragraph,  based on the mentioned identification $Sp(4,\F_2)\simeq \mathfrak{S}_6$, actually proves the 
following stronger result: 
\smallskip

\begin{theorem}\label{T.tre}
 The matrix $T\in Sp(4,\Z)$ is not contained in the commutator subgroup (the group generated by all the commutators) of $Sp(4,\Z)$.
 \end{theorem}
\smallskip

Again, we prove the   even  stronger result   that  the reduction modulo $2$ of $T$ in $Sp(4,\F_2)$ does not belong to the   commutator subgroup.

\smallskip

 We describe now  more formally the isomorphism  $\mathfrak{S}_6\simeq Sp(4,\F_2)$  following \S  10.1.12 of Serre's book  \cite{serre2}.

Let $H\subset\F_2^6$ be the hyperplane of equation $\sum_{i=1}^6 x_i=0$. 
Consider the alternating bilinear form $H\times H\to \F_2$ sending $(x,y)\mapsto \sum_{i}x_iy_i$. The vector $(1,\ldots,1)$ is orthogonal to the whole space, and the induced bilinear form on the four dimensional vector space $V:=H/<(1,\ldots,1)>$ turns out to be non degenerate.
The group $\mathfrak{S}_6$ acts naturally on $\F_2^6$ leaving $H$ invariant; also it conserves the bilinear form and fixes the point $(1,\ldots,1)$, so it acts (faithfully) on $V$ as a group of symplectic automorphisms. Hence we obtain an embedding $\mathfrak{S}_6\hookrightarrow Sp(V)=Sp(4,\F_2)$. To prove that this embedding is in fact an isomorphism, we compare the orders of the two groups.

The order of $Sp(4,\F_2)$ can be computed as follows: the set of non-degenerate planes in $\F_2^4$ has cardinality $15\cdot 8 /6=20$, since one can choose a non-zero vector $v_1$ in $15$ ways and a second vector $v_2\in\F_2^4\setminus v_1^\perp$ in $16-8=8$ ways. Hence there are $15\cdot 8$ possibility for the ordered base $(v_1,v_2)$ and each plane admits six order bases, hence the cardinality of the set of non-degenerate planes is $20$.
The group $Sp(4,\F_2)$ acts transitively on the set of non-degenerate planes and the stabilizer of any such  plane is  isomorphic to $\SL_2(\F_2)\times\SL_2(\F_2)\simeq \mathfrak{S}_3\times \mathfrak{S}_3$, so has order $36$. It follows that 
\begin{equation*}
|Sp(4,\F_2)|=20\times 36=720=6!
\end{equation*}
We then obtain the sought isomorphism $\mathfrak{S}_6\simeq Sp(4,\F_2)$.
 
We want to prove that the matrix $T$ corresponds, via this isomorphism, to an odd permutation in $\mathfrak{S}_6$, hence it does not belong to the derived subgroup of $Sp(4,\F_2)$. 

Now, $T$ has order two, and every even permutation  of order two in $S_6$ is conjugate  to the permutation $(1,2)\circ (3,4)$. To prove the theorem, it then suffices to show that this permutation gives rise to a matrix in $Sp(4,\F_2)$  which is not conjugate to $T$.

The quotient space $V=H/<(1,\ldots,1)>$ is represented by the vectors $(x_1,\ldots,x_6)$ with vanishing last coordinate $x_6$ and vanishing sum of the coordinates.
A basis is provided by $v_1=(1,0,0,0,1,0), v_2=(0,1,0,0,1,0), v_3=(0,0,1,0,1,0)$ and $v_4=(0,0,0,1,1,0)$.

The permutation $(1,2)\circ (3,4)$ sends
\begin{equation*}
\begin{matrix}
v_1 &\leftrightarrow & v_2\\
v_3 &\leftrightarrow  &v_4
\end{matrix}
\end{equation*}
hence corresponds to the matrix  
\begin{equation*}
S=\left(\begin{matrix}0&1&0&0\\ 1&0&0&0\\ 0&0&0&1\\ 0&0&1&0\end{matrix}\right),
\end{equation*}
 which is not conjugate, not even in $\SL_4(\F_2)$, to the matrix $T$ (compare the ranks of $T+I$ and $S+I$).
This ends the proof of Theorem \ref{T.tre}.

\medskip
Actually, it turns out that the matrix $T$ corresponds  to a  permutation of $\mathfrak{S}_6$ conjugate  to $(1,2)\circ (3,4)\circ (5,6)$.

\smallskip

Note that if we consider the natural action of $Sp(4,\F_2)$ on the set $\F_2^4$, then the matrix $T$ induces an even permutation of the fifteen non-zero vectors of $\F_2^4$ (while it fixes the origin), so we could not prove  that it is not a commutator simply by looking at the natural embedding $Sp(4,\F_2)\hookrightarrow \mathfrak{S}_{15}$.

\smallskip
\section{Commutators in the Symplectic group $Sp (2g, \ZZ),\  g \geq 3$}

We now show: 

\begin{theo}
 In  every dimension $2g$ with $g\geq 3$, for every $m\geq 0$  
there exist symplectic matrices $X,Y\in Sp(2g,\Z)$ whose commutator equals 
\begin{equation*}
T^m=\left(\begin{matrix}
I+m\Delta &0&0&\ldots &0\cr
0&I&0 &\ldots &0\cr
\vdots&\vdots&\ddots& I&0\cr
0&0&\ldots&0&I
\end{matrix}\right)
\end{equation*}
\end{theo}
Here, as before,  $\Delta$ is the matrix $\left(\begin{matrix}0&1\cr 0&0\end{matrix}\right).$

\begin{proof}
Clearly, it suffices to prove this statement in the case $g=3$, i.e. for $Sp(6,\Z)$.

Here is a concrete example:

\begin{equation*}
X=\left(\begin{matrix}
1&0&1&0&0&1\cr
0&1&0&0&0&0\cr
0&0&1&0&0&0\cr
0&0&0&1&1&0\cr
0&1&0&0&1&0\cr
0&0&1&0&0&1\cr
\end{matrix}\right),\qquad 
Y=\left(\begin{matrix}
1&0&0&m&0&0\cr
0&1&0&0&0&0\cr
0&m&1&0&0&0\cr
0&0&0&1&0&0\cr
0&0&0&0&1&0\cr
0&0&0&0&m&1\cr
\end{matrix}\right).
\end{equation*}
\end{proof}

We note that both $X$ and $Y$ are unipotent. In general, we can prove that for every symplectic solution $(X,Y)$ of $XYX^{-1}Y^{-1}=T$, both $X$ and $Y$ must have an eigenvalue equal to $\pm 1$.

\medskip

\begin{remark}
We want now to  show that without the hypothesis that the matrices be symplectic, we have examples even in dimension $\leq 4$. For instance, the following pair of (unipotent) matrices in $\SL_3(\Z)$ 
\begin{equation*}
X= \left(\begin{matrix}1&0&1\cr 0&1&0\cr 0&0&1
\end{matrix}\right),\qquad
Y= \left(\begin{matrix}1&0&0\cr 0&1&0\cr 0&m&1
\end{matrix}\right)
\end{equation*}
provides a solution to the equation 
\begin{equation*}
XYX^{-1}Y^{-1}= \left(\begin{matrix}1&m&0\cr 0&1&0\cr 0&0&1
\end{matrix}\right)=T^m.
\end{equation*}

\end{remark}

\subsection{The case $g \geq 3$, more general.}

For the sake of simplicity, we introduce the following notation:

\begin{definition}
Using the standard inclusion of $Sp (2g', \ZZ) \subset Sp (2g, \ZZ)$ for $g \geq g'$,
we define $T_1$ as the image of  the matrix
 \begin{equation*}
T:=\left(\begin{matrix}
1&1\cr 0&1
\end{matrix}\right),
\end{equation*}
in every $Sp (2g, \ZZ)$.

We define, for $ g \geq 2$, $T_2$ as the image of 
 the matrix 
\begin{equation*}
T_2=\left(\begin{matrix}
1&1&0&0\cr 0&1&0&0\cr 0&0&1&1\cr 0&0&0&1
\end{matrix}\right)
\end{equation*}

and we define similarly $T_k \in Sp (2g, \ZZ)$, for $ g \geq k$.

\end{definition}

We have the following 

\begin{theo}
$T_2$ is always a commutator ($g\geq 2$); also $T_3$ is always a commutator ($g\geq 3$). 

\end{theo}
\begin{proof}
Here is an explicit solution for $T_2$: $XYX^{-1}Y^{-1}=T_2$ where
\begin{equation*}
X=\left(\begin{matrix}
0&0&1&0\\ 0&1&0&1\\ 1&0&-1&0\\ 0&1&0&0
\end{matrix}\right), \qquad 
Y=\left(\begin{matrix}
1&-1&0&-1\\ 0&1&0&0\\ 0&-1&1&0\\ 0&0&0&1
\end{matrix}\right).
\end{equation*}

For $T_3$ we have: $T_3 = XY X^{-1} Y^{-1}$ with

\begin{equation*}
X= \left( \begin{matrix} 0&0&1&0&0&0\cr 0&0&0&1&0&1\cr 1&0&1&0&-1&0\cr
0&1&0&0&0&0\cr 0&0&-1&0&1&0\cr 0&1&0&0&0&1\end{matrix}\right), \qquad
Y=\left(\begin{matrix}  1&-2&0&1&0&0\cr 0&1&0&0&0&0\cr 0&1&1&-1&0&-1\cr
0&0&0&1&0&0\cr 0&0&0&-1&1&-1\cr 0&0&0&0&0&1 \end{matrix}\right)
\end{equation*}

\end{proof}

The idea for constructing $X,Y$ comes from the following remark: for every non-zero complex number $\lambda$, setting
\begin{equation*}
A=\left(\begin{matrix} \lambda&0\\ 0&\lambda^{-1}\end{matrix}\right), \qquad B=\left(\begin{matrix}1&(\lambda^2-1)^{-1}\\ 0&1 \end{matrix}\right)
\end{equation*}
we have
\begin{equation*}
ABA^{-1}B^{-1}=T_1.
\end{equation*}
We then look for a number field $\mathbb{K}$ containing a unit (of its ring of integers) $\lambda$ such that $\lambda^2-1$ is also a unit. Letting $n=[\mathbb{K}:\QQ]$ be its degree, we can view $\mathbb{K}^2$ as a vector space of dimension $2n$ over $\QQ$. The two matrices $A,B$ defined above induce automorphisms of this vector space, and in a suitable basis define two matrices $\tilde{A},\tilde{B}\in\SL_{2n}(\Z)$ satisfying $[A,B]=T_n$. The problem is defining a symplectic form $\mathbb{K}^2\to \QQ$ inducing the standard one on $\QQ^{2n}$, after identification $\mathbb{K}^2\simeq \QQ^{2n}$. 

It turns out that for $n=2$ there is only one choice for the number field $\mathbb{K}$, namely the field $\QQ(\lambda)$ where $\lambda$ is the `golden ratio' satisfying $\lambda^2=1+\lambda$. Identifying $\mathbb{K}^2\simeq \QQ^4$ via the basis ${1 \choose 0}, {0\choose 1}, {\lambda\choose 0}, {0\choose \lambda}$ we obtain from $A,B$ the matrices $X,Y$ of the theorem.

\smallskip

For $n=3$, again we have only one choice for the cubic number field, namely the field $\QQ(\lambda)$ where
\begin{equation*}
\lambda^3= 2\lambda^2 +\lambda-1.
\end{equation*}
The basis to be used to identify $\mathbb{K}^2$ with $\QQ^6$ is 
\begin{equation*}
{1\choose 0}, {0\choose 1},{\lambda\choose 0}, {0\choose\lambda},{1+\lambda-\lambda^2\choose 0}, {0\choose 1+\lambda-\lambda^2}.
\end{equation*}
Again, the matrices $X,Y$ are then obtained from the action of $A,B$ on $\mathbb{K}^2\simeq \QQ^6$.
\medskip

We can now obtain a more general result  using our previous results and the following:

\begin{rem}\label{sum}
(1) Assume that $A_i \in Sp (2g_i, \ZZ)$ is a commutator, for $i=1,2$. 

Then $A_1 \oplus A_2 \in Sp (2(g_1 + g_2), \ZZ)$ is also commutator.

(2) In particular, this holds for $A_2$  equal to the identity matrix.

\end{rem}

 Summarising, we obtained the following

\begin{theo}\label{last}

\begin{enumerate}
\item
$T_1^m$ is never a commutator for $g=1$.
\item
$T_1^m$ is a commutator if and only if $ m$ is even, for $g=2$.
\item
$T_1^m$ is always a commutator for $g \geq 3$.
\item
 $T_k$ is a commutator for all $ g \geq k \geq 2$.
 \item
  $T_k^m$ is a commutator for all $ g \geq 2k $, when $m$  is even.
 \item
 $T_k^m$ is a commutator for all $ g \geq 3k $, when $m$  is odd.

\end{enumerate}

\end{theo}

{\bf Acknowledgements. }  The present research was motivated by some question posed by Domingo Toledo
to the first named author concerning the singular fibres of the Cartwright-Steger surface, which led to  an intense
 exchange and  (still unpublished) joint  work on the geometric construction of this surface.
The first author is therefore  much   indebted to Domingo Toledo and Matthew Stover for very useful  communications.
Thanks also to Daniele Zuddas for pointing out the reference \cite{matsumoto}   and for very useful conversations
concerning the Polizzi-Rito-Roulleau surface. 
 

\bigskip

\address{Fabrizio Catanese\\ 
Lehrstuhl Mathematik VIII\\
Mathematisches Institut der Universit\"at Bayreuth\\
NW II,  Universit\"atsstr. 30\\
95447 Bayreuth - Germany\\}
\email{fabrizio.catanese@uni-bayreuth.de}

\bigskip

\address{Pietro Corvaja\\
Dipartimento di Scienze Matematiche,
Informatiche e Fisiche\\ 
via delle Scienze, 206\\
33100 Udine - Italy\\}
\email{pietro.corvaja@uniud.it}

\bigskip

\address{Umberto Zannier\\
Scuola Normale Superiore\\
Piazza dei Cavalieri, 7\\
56126 Pisa - Italy\\
}
\email{u.zannier@sns.it}

\end{document}